\documentclass[12pt]{article}
\usepackage{amsmath,amssymb,latexsym,color,amsfonts,pdfsync,amsthm}
\usepackage{url,graphicx}

\usepackage{float}
\restylefloat{table}

\newtheorem{theorem}{Theorem}[section]

\newtheorem{lemma}[theorem]{Lemma}

\newtheorem{remark}[theorem]{Remark}
\newtheorem{definition}[theorem]{Definition}

\newcommand{\Proof}{ \noindent{\bf Proof:}\quad }

\def\F{\mathbb{F}}
\def\Fq{\mathbb{F}_q}
\def\PG{\mathrm{PG}}

\def\GL{\mathrm{GL}}

\def\RR{\mathbb{R}}

\begin{document}

\title{Classification of subspaces in ${\mathbb{F}}^2\otimes {\mathbb{F}}^3$ and orbits in ${\mathbb{F}}^2\otimes{\mathbb{F}}^3 \otimes {\mathbb{F}}^r$}
\author{Michel Lavrauw\footnote{The research of the first author was supported by the Fund for Scientific Research - Flanders (FWO) and by a Progetto di Ateneo from Universit\`a di Padova (CPDA113797/11).}
 and John Sheekey\footnote{The second author acknowledges the support of the Fund for Scientific Research - Flanders (FWO).}
}
\date{\today}
\maketitle

\begin{abstract}
This paper contains the classification of the orbits of elements of the tensor product spaces ${\mathbb{F}}^2\otimes{\mathbb{F}}^3 \otimes {\mathbb{F}}^r$, $r\geq 1$, under the action of two natural groups, for all finite; real; and algebraically closed fields. For each of the orbits we determine: a canonical form; the tensor rank; the rank distribution of the contraction spaces; and a geometric description. The proof is based on the study of the contraction spaces in ${\mathrm{PG}}({\mathbb{F}}^2\otimes {\mathbb{F}}^3)$ and is geometric in nature. Although the main focus is on finite fields, the techniques are mostly field independent.\footnote{MSC: 05B25; 05E20; 15A69; 51E20}
\end{abstract}

%%%%%%%%%%%%%%%%%%%%%%%%%%%%%%%%%%%%%%%%%%%%%%%%%%%%%%%%%%%%%%%%%%%%%%%%%%%%%
%%   INSERT TEXT OF ARTICLE                                                %%
%%%%%%%%%%%%%%%%%%%%%%%%%%%%%%%%%%%%%%%%%%%%%%%%%%%%%%%%%%%%%%%%%%%%%%%%%%%%%

\section{Introduction and Motivation}
Tensors are fundamental objects in both algebra and geometry. They also have many important applications, for example in complexity theory \cite{BuClSh1997}, \cite{Blaser2004}, quantum information and entanglement \cite{Gurvits04}, \cite{Heydari2008}, \cite{LaQiYe2012}, and quantum coding \cite{GlGuMaGu2006}. There is a vast literature on various topics in the theory of tensors. See the recent book of Landsberg \cite{Landsberg2012} for details on many of these.

Much of this literature concerns tensors over fields of characteristic zero, and/or algebraically closed fields. However, tensors over finite fields are also of great interest, due for example to their connections to complexity theory, and finite semifields \cite{Lavrauw2012}.

In small dimensions, there has been much work on classifying tensors, mainly over the complex numbers. See for example \cite{ThCh1938}, \cite{Nurmiev2000}, where $3\times 3\times 3$ tensors over the complex numbers are studied.

Deep geometric analysis of tensors over binary fields was carried out in \cite{HaOdSa2012}, \cite{ShGoHa2012}. Recent computational result over some small finite fields can be found in e.g. \cite{BrHu2012}, \cite{BrSt2013}, \cite{ShGoHa2012}. 

It is worth noting that the theory of Weierstrass-Kronecker pencils gives a way to approach the classification of tensors, for example as in \cite{JaJa1979}. However, this approach has some downsides and is not sufficient to complete the classification, as explained in \cite[Section 1.2]{LaSh233}.

In this paper we take an elementary, field-independent approach. We obtain a full classification of subspaces in $\F^2 \otimes \F^3$, and provide tables containing representatives for each orbit. This classification of subspaces is based on the classification of tensors in $\F^2\otimes \F^3\otimes \F^3$, obtained in \cite{LaSh233} using elementary geometric methods, and is independent of any other previously obtained (partial) classification. With this classification at hand, we are also able to enumerate the orbits of tensors in $\F^2 \otimes \F^3\otimes \F^r$ for all $r\geq 1$. Most importantly, our approach gives geometric insight into the nature of the orbits of subspaces and tensors. This geometric insight is very useful in the area of Finite Geometry, and to our knowledge cannot be found anywhere else.

\section{Preliminaries}

This paper follows on from \cite{LaSh233}, where canonical forms of $2\times 3 \times 3$ tensors are obtained, and to which we refer for notation and terminology. Here we restrict ourselves to a brief description of the necessary background for our study. Given a tensor product $V=V_1\otimes V_2\otimes V_3$, we consider two natural groups: the stabiliser $G$ in $\GL(V)$ of the set of fundamental tensors in $V$ and the subgroup $H\leq G$ isomorphic to $\GL(V_1)\times \GL(V_2) \times \GL(V_3)$. The problem is to determine the orbits of $G$ and $H$ on $V$. Previously obtained results for specific fields or concerning related problems can be found in for instance \cite{BrHu2012}, \cite{BrSt2013}, \cite{LaSh2014}, \cite{ShGoHa2012}, \cite{Stavrou2013}.
In \cite{LaSh233}, the following was shown.

\begin{theorem}\label{thm:233}
If $\F$ is a finite field, then there are precisely 21 $H$-orbits and 18 $G$-orbits of tensors in $\F^2 \otimes \F^3 \otimes \F^3$. For any algebraically closed field $\F$, there are precisely 18 $H$-orbits and 15 $G$-orbits of tensors in $\F^2 \otimes \F^3 \otimes \F^3$. There are precisely 20 $H$-orbits and 17 $G$-orbits of tensors in $\RR^2 \otimes \RR^3 \otimes \RR^3$.
\end{theorem}

Four convenience we include the canonical forms for the 18 $G$-orbits in the finite field case, where we assume that $e_1,e_2,e_3$ form a basis for $\F^3$ and $e=e_1 \otimes e_1+ e_2\otimes e_2 + e_3 \otimes e_3$.

\begin{tabular}{lll}
$o_0$& 0& \\
 $o_1$& $ e_1 \otimes e_1 \otimes e_1 $&\\
$o_2$& $ e_1 \otimes (e_1 \otimes e_1+ e_2\otimes e_2) $&\\
$o_3$& $ e_1 \otimes e $&\\
$o_4$& $ e_1 \otimes e_1 \otimes e_1+ e_2\otimes e_1 \otimes e_2  $&\\
 $o_5$& $ e_1 \otimes e_1 \otimes e_1+ e_2\otimes e_2 \otimes e_2  $&\\
 $o_6$ & $ e_1 \otimes e_1 \otimes e_1+ e_2\otimes (e_1 \otimes e_2 + e_2 \otimes e_1)  $&\\
 $o_7$& $ e_1 \otimes e_1 \otimes e_3+ e_2\otimes (e_1 \otimes e_1 + e_2 \otimes e_2)  $&\\
 $o_8$& $ e_1 \otimes e_1 \otimes e_1+ e_2\otimes (e_2 \otimes e_2 + e_3 \otimes e_3)  $&\\
$o_9$& $ e_1 \otimes e_3 \otimes e_1+ e_2\otimes e  $&\\
$o_{10}$&  $ e_1\otimes (e_1\otimes e_1+ e_2\otimes e_2+u e_1\otimes e_2) +  
e_2\otimes (e_1\otimes e_2+v e_2\otimes e_1),$&\\
 & $v\lambda^2+uv\lambda - 1 \neq 0$ for all
$\lambda \in \F$& \\
$o_{11}$& $ e_1\otimes (e_1 \otimes e_1 + e_2 \otimes e_2)+ 
e_2\otimes (e_1 \otimes e_2 + e_2 \otimes e_3)$&\\
$o_{12}$& $ e_1\otimes (e_1 \otimes e_1 + e_2 \otimes e_2)+ 
e_2\otimes (e_1 \otimes e_3 + e_3 \otimes e_2)$&\\
$o_{13}$& $ e_1\otimes (e_1 \otimes e_1 + e_2 \otimes e_2)+ 
e_2\otimes (e_1 \otimes e_2 + e_3 \otimes e_3)$&\\
$o_{14}$& $ e_1\otimes (e_1 \otimes e_1 + e_2 \otimes e_2)+ e_2\otimes (e_2 \otimes e_2 + e_3 \otimes e_3)$&\\
$o_{15}$&  $ e_1\otimes (e+u e_1\otimes e_2) + 
e_2\otimes (e_1\otimes e_2+v e_2\otimes e_1),$&\\
 & $v\lambda^2+uv\lambda - 1 \neq 0$ for all
$\lambda \in \F$& \\
$o_{16}$& $ e_1\otimes e+ e_2\otimes (e_1 \otimes e_2 + e_2 \otimes e_3)$&\\
$o_{17}$& $ e_1\otimes e+ 
e_2\otimes (e_1\otimes e_2 + e_2\otimes  e_3 + e_3\otimes (\alpha e_1 + \beta e_2 + \gamma e_3)),$ &\\
 & $\lambda^3+\gamma \lambda^2- \beta \lambda+ \alpha \neq 0$ for all $\lambda \in \F$ & \\
\end{tabular}

The three extra $H$-orbits are $o_4^T$, $o_7^{T}$, and $o_{11}^T$, where $T$ is the linear map defined by 
\begin{eqnarray}\label{eqn:transpose}
(a \otimes b \otimes c)^T := a \otimes c \otimes b.
\end{eqnarray}

In this paper we extend this classification result in various directions. Firstly, we determine the orbits on the points, lines and planes in 
$\PG(\F^2 \otimes \F^3)$, including a geometric description and a canonical form for each orbit (see Tables \ref{tab:points}, \ref{tab:lines}, \ref{tab:planes}). In Section \ref{sec:sols+hyps} we determine the orbits of the remaining subspaces in $\PG(\F^2 \otimes \F^3)$, see Table \ref{tab:solids} and Table \ref{tab:hyp}. In Section \ref{sec:summary_r} we extend the classification from \cite{LaSh233} (Theorem \ref{thm:233}) to orbits of tensors in $\F^2 \otimes \F^3\otimes \F^r$, for any $r\geq 1$. Finally in Section \ref{sec:rank} we determine the tensor rank of each orbit in $\F^2 \otimes \F^3\otimes \F^r$.

\bigskip

We end the section with some necessary background.
Let $V=\bigotimes_{i=1}^r V_i$, where $V_1, \ldots, V_r$ are finite dimensional vector spaces over some field $\F$, with $\dim V_i=n_i<\infty$. The set of {\it fundamental tensors} is the set $\{v_1\otimes v_2 \otimes \ldots \otimes v_r :  v_i\in V_i\backslash\{0\}\}$. Projectively, this set corresponds to points on the {\it Segre variety} $S_{n_1,n_2,\ldots,n_r}$, that is the image of a {\it Segre embedding} $\sigma_{n_1,\ldots,n_r}$.
The {\it rank} of a tensor $A$, is defined to be the minimum number $k$ such that there exist fundamental tensors $\alpha_1,\ldots,\alpha_k$ with $A \in \langle \alpha_1,\ldots,\alpha_k\rangle$. 
For $A\in V_1\otimes V_2 \otimes V_3$ we define the {\it first contraction space of $A$} as the subspace 
$A_1:=\langle w_1^\vee(A)~:~w_1^\vee \in V_1^\vee\rangle$ of $V_2\otimes V_3$, where $V_1^\vee$ denotes the dual space of $V_1$, and the {\it contraction} $w_1^\vee(A)$ is defined by its action on the fundamental tensors:
$w_1^\vee(v_1\otimes v_2 \otimes v_3)=w_1^\vee(v_1)v_2\otimes v_3$.
Similarly we define the {\it second} and {\it third contraction space}, and denote these 
by $A_2$ and $A_3$, respectively. We will consider the projective subspaces $\PG(A_i)$ of $\PG(V_j \otimes V_k)$, where $j<k$ and $\{i,j,k\}=\{1,2,3\}$.
The setwise stabilizer of the set of rank one tensors in the contracted space $V_j \otimes V_k$ will be denoted by $G_i$, where $j<k$ and $\{i,j,k\}=\{1,2,3\}$, and the subgroups $\GL(V_j)\times\GL(V_k)$ of $G_i$ by $H_i$. The corresponding projective groups are denoted by $\bar{G}_i$ and $\bar{H}_i$.
The {\it ($i$-th) rank distribution} $r_i(A)$ of a tensor $A$ is defined to be the tuple whose $j$-th entry is the number of rank $j$ points in the $i$-th contraction space $\PG(A_i)$. 
There are two types of lines on the Segre variety $S_{2,3}$: those of the form $\sigma(x \times \ell)$ for a point $x$ in $\PG(\F^2)$ and a line $\ell$ in $\PG(\F^3)$, which we refer to as {\it lines of the first kind}; and those of the form 
$\sigma(\PG(\F^2) \times y)$ for a point $y$ in $\PG(\F^3)$, which we refer to as {\it lines of the second kind}.

\section{Classification of points, lines and planes in $\PG(\F^2 \otimes \F^3$)}
Note that the orbits under $H_2$ on subspaces of $\F^2\otimes \F^3\setminus \{0\}$ are equivalent to the orbits
of the projective group $\bar{H_2}$ induced by $H_2$ on subspaces of the projective space $\PG(\F^2\otimes \F^3)$.
Using the classification of tensors in $\F^2\otimes \F^3\otimes \F^3$, we can classify the points, lines and planes in $\PG(\F^2\otimes \F^3)$ up to ${\bar H}_2$-equivalence, by considering the second and/or third contraction spaces, cf. \cite[Lemma 2.1]{LaSh233}.

\begin{remark}
As we will see, for some of the cases, the second and third contraction spaces of a canonical form for $o_i$ in $\F^2\otimes \F^3 \otimes \F^3$ belong to different $\bar{H}_2$-orbits of subspaces in $\PG(\F^2\otimes\F^3)$. In this case we will represent the second contraction space by $o_i$ and the third by ${\mathbf{o_i^T}}$ in accordance with the notation used in \cite{LaSh233}.
\end{remark}

\subsection{Finite fields}
We first classify the $\bar{H}_2$-orbits of subspaces of $\PG(\F^2\otimes\F^3)$, where $\F$ is a finite field. 
The classification for other fields is based on this case.
\begin{theorem}\label{thm:pts_lines_planes}
If $\F$ is a finite field, then under the action of $\bar{H}_2$, there are 2 orbits of points, 7 orbits of lines and 11 orbits of planes in $\PG(\F^2\otimes \F^3)$. The description is as in Table \ref{tab:points}, Table \ref{tab:lines}, and Table \ref{tab:planes}.
\end{theorem}
\Proof For each $A$ in the list of canonical forms in the classification of $2\times 3 \times 3$ tensors (Theorem \ref{thm:233}) we take the second contraction space $\PG(A_2)$ and third contraction space $\PG(A_3)$. It is clear that there are two orbits of points in $\PG(\F^2\otimes\F^3)$: 
points of rank one and points of rank two.
Both the second and third contraction space of the canonical form $A= e_1 \otimes e_1 \otimes e_1$ of ${\mathbf{o_1}}$ gives $\langle e_1\otimes e_1\rangle$, which amounts to a point of rank one. The points of rank two are represented by the second contraction space of the canonical form of the orbit $o_4$. Note that the third contraction space of $A$ in ${\mathbf{o_4}}$ gives $\langle e_1\otimes e_1, e_2\otimes e_1\rangle$, corresponding to a line of the first kind on $S_{2,3}$. As mentioned above, we represent this orbit by ${\mathbf{o_4^T}}$.

The second and third contraction spaces of the canonical form of ${\mathbf{o_2}}$ both give  $\langle e_1 \otimes e_1, e_1\otimes e_2\rangle$, which gives a line of the second kind on the Segre variety $S_{2,3}$.

For the orbit ${\mathbf{o_3}}$ we get $\langle e_1 \otimes e_1, e_1\otimes e_2,e_1\otimes e_3\rangle$
giving a plane of the second kind on $S_{2,3}$.

The next orbits corresponding to canonical forms for ${\mathbf{o_5}}$, ${\mathbf{o_6}}$, and ${\mathbf{o_7}}$ are represented by 
$\langle e_1\otimes e_1,e_2\otimes e_2\rangle$,
$\langle e_1\otimes e_1 + e_2\otimes e_2, e_2\otimes e_1\rangle$, and 
$\langle e_1\otimes e_3 + e_2 \otimes e_1, e_2\otimes e_2\rangle$, which give
a secant line contained in an $\langle S_{2,2}\rangle$, a tangent line contained in an $\langle S_{2,2}\rangle$ and
a tangent line not contained in an $\langle S_{2,2}\rangle$, respectively.

Note that the third contraction spaces of $o_5$ and $o_6$ give the same orbits as the corresponding second contraction spaces, while the third contraction space of the canonical form for $o_7$ gives 
$\langle e_2\otimes e_1,e_2\otimes e_2,e_1\otimes e_1\rangle$ which amounts to a plane contained in $\langle S_{2,2}\rangle$, and intersecting $S_{2,2}$ in a line of the first kind and a line of the second kind. This orbit is represented by ${\mathbf{o_7^T}}$.

The second and third contraction spaces for ${\mathbf{o_8}}$ both give 
$\langle e_1\otimes e_1,e_2\otimes e_2,e_2\otimes e_3\rangle$, inducing a plane intersecting $S_{2,3}$ in a line of the second kind $\sigma_{2,3}(\langle e_2\rangle \times \langle e_2,e_3\rangle)$ and a point $\langle e_1 \otimes e_1\rangle$. Note that this plane is not contained in an $\langle S_{2,2}\rangle$.

The second contraction space for ${\mathbf{o_9}}$ gives
$\langle e_2\otimes e_1,e_2\otimes e_2,e_1\otimes e_1+ e_2\otimes e_3\rangle$ giving a plane intersecting $S_{2,3}$ in a line of the second kind. Again this plane is not contained in an $\langle S_{2,2}\rangle$.
The third contraction space gives $\langle e_2\otimes e_3,e_2\otimes e_2,e_1\otimes e_3+ e_2\otimes e_1\rangle$, which is different but equivalent under the element $(g,h)\in H_2$, where $g$ is the identity and $h$ is the basis transformation $e_1\mapsto e_3$, $e_2\mapsto e_2$, and $e_3\mapsto e_1$ in the second factor.

Both the second and the third contraction spaces of ${\mathbf{o_{10}}}$ give a constant rank two line in $\PG(\F^2\otimes\F^3)$ contained in an $\langle S_{2,2}\rangle$. There is one such orbit of lines under the action of $H_2$, since there is one orbit of subvarieties $S_{2,2}$ and these lines correspond to the (unique) orbit of nonsingular points in the Segre variety product of three projective lines (see \cite{LaSh2014}).

The second contraction space for ${\mathbf{o_{11}}}$ gives
$\langle e_1\otimes e_1 + e_2\otimes e_2,e_1\otimes e_2+e_2\otimes e_3\rangle$
inducing a constant rank two line not contained in an $\langle S_{2,2}\rangle$, while the third contraction space
gives the plane defined by
$\langle e_1\otimes e_1,e_1\otimes e_2+e_2\otimes e_1,e_2\otimes e_2\rangle$. This plane is contained in an $\langle S_{2,2}\rangle$ and intersects $S_{2,2}$ in a conic. The corresponding orbit is denoted by ${\mathbf{o_{11}}^T}$.

From the canonical form $A$ for ${\mathbf{o_{12}}}$ we obtain 
$A_2=\langle e_1\otimes e_1+ e_2\otimes e_3,e_1\otimes e_2,e_2\otimes e_2\rangle$. Then
$\PG(A_2)$ is a plane not contained in an $\langle S_{2,2}\rangle$ and intersecting $S_{2,3}$ in a  line of the first kind. The third contraction space is 
$A_3=\langle e_1\otimes e_1, e_1\otimes e_2+e_2\otimes e_3,e_2\otimes e_1\rangle$. 
This is equivalent to the second contraction space under the element $(g,h)\in H_2$, where $g$ is the identity and $h$ is the basis transformation $e_1\mapsto e_2$, $e_2\mapsto e_1$, and $e_3\mapsto e_3$ in the second factor.

The second contraction space for ${\mathbf{o_{13}}}$ gives 
$A_2=\langle e_1\otimes e_1+e_2\otimes e_2,e_1\otimes e_2,e_2\otimes e_3\rangle$, inducing the plane $\PG(A_2)$ not contained in an $\langle S_{2,2}\rangle$, and intersecting $S_{2,3}$ in exactly two points. Again the third contraction space is equivalent to the second contraction space under the element $(g,h)\in H_2$, where $g$ is the identity and $h$ is the basis transformation $e_1\mapsto e_2$, $e_2\mapsto e_1$, and $e_3\mapsto e_3$ in the second factor.

Both the second and the third contraction spaces of the canonical form $A$ for ${\mathbf{o_{14}}}$ give the plane $\PG(A_2)$, defined by $A_2=\langle e_1\otimes e_1,(e_1+e_2)\otimes e_2,e_2\otimes e_3\rangle$, which is not contained in an $\langle S_{2,2}\rangle$ and intersects $S_{2,3}$
in exactly three points.

For ${\mathbf{o_{15}}}$ the canonical form $A$ gives the second contraction space
$A_2=\langle e_1\otimes (e_1+u e_2)+ e_2\otimes e_2,e_1\otimes e_2+ve_2\otimes e_1,e_1\otimes e_3\rangle$. This corresponds to a plane $\PG(A_2)$ intersecting $S_{2,3}$ in exactly one point.
For ${\mathbf{o_{16}}}$ the canonical form $A$ gives the second contraction space
$A_2=\langle e_1\otimes e_1+ e_2\otimes e_2,e_1\otimes e_2+e_2\otimes e_3,e_1\otimes e_3\rangle$, again amounting to a plane intersecting $S_{2,3}$ in exactly one point. Both of these are equivalent to the third contraction space of the corresponding orbit by 
the element $(g,h)\in H_2$, where $g$ is the identity and $h$ is the appropriate basis transformation in the second factor space ($e_1\mapsto e_2$, $e_2\mapsto e_1$, $e_3\mapsto e_3$ for $o_{15}$ and $e_1\mapsto e_3$, $e_2\mapsto e_3$, $e_3\mapsto e_1$ for $o_{16}$).

Both planes contain a unique line $l$ for which it holds that for each rank two point $y$ on $l$, the solid $\langle Q(y)\rangle$ (see \cite[Lemma 2.4]{LaSh233}) meets the plane in $l$. Moreover, for each other rank two point $y$ in the plane, the solid $\langle Q(y)\rangle$ meets the plane in $y$.
The geometric characterisation of these two planes is then that for $o_{15}$, this line is a line of type $o_{10}$ (and hence does not go through the unique point of rank one in that plane), while for $o_{16}$, this line is of type $o_6$ and contains the unique point of rank one in that plane.

Finally for ${\mathbf{o_{17}}}$, we recall that the first contraction space gives a line $\PG(A_1)$ which has no singular points. In other words, for each contraction $w_1^\vee \neq 0$ in the first factor, $w_1^\vee(A)$ is a nonsingular $3 \times 3$ tensor. Equivalently, for $i\in \{2,3\}$, and  for each 
two nonzero contractions $w_1^\vee \in V_1^\vee$, $w_i^\vee \in V_i^\vee$
the double contraction $w_i^\vee(w_1^\vee(A))$ is a nonzero vector. 
Since $w_i^\vee(w_1^\vee(A))=w_1^\vee(w_i^\vee(A))$, this implies that each $w_i^\vee(A)$ must have rank two.
It follows that both the second and third contraction space
 gives a constant rank two plane.
 \qed
 
The $\bar{H}_2$-orbits of points, lines and planes in $\PG(\F^2\otimes \F^3)$ for $\F=\F_q$ are summarised in the following tables. The dimensions of the minimal Segre variety spanning a subspace containing a representative of the orbit is given in the third column, e.g. $2\times 1$ means that there exists a variety $X\cong S_{2,1}$ contained in the Segre variety $S_{2,3}$, such that the representative of the orbit is contained in $\langle X \rangle$. 
The last column represents the subspace.

\begin{table}[h]
\begin{center}
\begin{tabular}{|c|c|c|c|}
\hline
{\bf Orbit} & {\bf Intersection with} $S_{2,3}$ &{\bf Min}%& $\mathbf{r_2(A)}$ 
& {\bf Rep}\\
\hline
$o_1$ & a point &$1 \times 1$%&$[1,0]$ 
& {\tiny{$\left [ \begin{array}{ccc} 1& \cdot & \cdot \\ \cdot& \cdot & \cdot\end{array}\right ]$}}\\ 
\hline
$o_4$ & $\emptyset$ &$1 \times 2$%&$[0,1]$ 
& {\tiny{$\left [ \begin{array}{ccc} 1& \cdot & \cdot \\ \cdot& 1 & \cdot \end{array}\right ]$}}\\ 
\hline
\end{tabular}
\caption{The two $\bar{H}_2$-orbits of points in $\PG(\F^2\otimes\F^3)$ for $\F=\F_q$.}
\label{tab:points}
\end{center}
\end{table}

\begin{table}[h]
\begin{center}
\begin{tabular}{|c|c|c|c|}
\hline
{\bf Orbit} & {\bf Intersection with} $S_{2,3}$ &{\bf Min}%& $\mathbf{r_2(A)}$ 
& {\bf Rep}\\
\hline
$o_2$ & line of the second kind &$1 \times 2$ %&$[q+1,0]$
& {\tiny{$\left [ \begin{array}{ccc} x& y & \cdot\\ \cdot& \cdot & \cdot\end{array}\right ] $}}\\ 
\hline
$o_4^T$ & line of the first kind  &$2 \times 1$%&$[q+1,0]$
& {\tiny{$\left [ \begin{array}{ccc} x& \cdot & \cdot \\y& \cdot & \cdot \end{array}\right ] $}}\\ 
\hline
$o_5$ & 2 points &$2 \times 2$%&$[2,q-1]$ 
& {\tiny{$\left [ \begin{array}{ccc} x& \cdot & \cdot \\ \cdot& y & \cdot \end{array}\right ] $}}\\ 
\hline
$o_6$ & 1 point &$2 \times 2$%&$[1,q]$
& {\tiny{$\left [ \begin{array}{ccc} x& \cdot & \cdot \\ y& x & \cdot \end{array}\right ] $}}\\ 
\hline
$o_7$ & 1 point &$2 \times 3$%&$[1,q]$
& {\tiny{$\left [ \begin{array}{ccc} x& y & \cdot \\ \cdot & \cdot  & x \end{array}\right ] $}}\\ 
\hline
$o_{10}$ & $\emptyset$ &$2 \times 2$%&$[0,q+1]$ 
& {\tiny{$\left [ \begin{array}{ccc} x& ux+y & \cdot \\ vy & x & \cdot  \end{array}\right ] $}}\\ 
 &&&{\tiny{$v\lambda^2+uv\lambda - 1 \neq 0, \forall \lambda \in \F$}}\\
\hline
$o_{11}$ &$\emptyset$ &$2 \times 3$%&$[0,q+1]$
& {\tiny{$\left [ \begin{array}{ccc} x& y & \cdot \\ \cdot & x  & y \end{array}\right ] $}}\\ 
\hline
\end{tabular}
\caption{The seven $\bar{H}_2$-orbits of lines in $\PG(\F^2\otimes\F^3)$  for $\F=\F_q$}
\label{tab:lines}
\end{center}
\end{table}

\begin{table}[H]
\begin{center}
\begin{tabular}{|c|c|c|c|}
\hline
{\bf Orbit} & {\bf Intersection with} $S_{2,3}$ &{\bf Min} %& $\mathbf{r_2(A)}$
& {\bf Rep}\\
\hline
$o_3$ & a plane  &$1 \times 3$ %&$[q^2+q+1,0]$
& {\tiny{$\left [ \begin{array}{ccc} x& y & z \\ \cdot & \cdot & \cdot \end{array}\right ] $}}\\ 
\hline
$o_7^T$ & 2 lines &$2 \times 2$ %&$[2q+1,q^2-q]$
& {\tiny{$\left [ \begin{array}{ccc} z& \cdot & \cdot \\x& y & \cdot \end{array}\right ] $}}\\ 
\hline
$o_8$ & a line and a point &$2 \times 3$ %&$[q+2,q^2-1]$
& {\tiny{$\left [ \begin{array}{ccc} x& \cdot & \cdot \\ \cdot & y & z \end{array}\right ] $}}\\ 
\hline
$o_9$ & a line of the second kind &$2 \times 3$ %&$[q+1,q^2]$
& {\tiny{$\left [ \begin{array}{ccc} z& \cdot & \cdot \\ x & y & z \end{array}\right ] $}}\\ 
\hline
$o_{11}^T$ &  a conic&$2 \times 2$ %&$[q+1,q^2]$
& {\tiny{$\left [ \begin{array}{ccc} x& y & \cdot \\ y & z & \cdot \end{array}\right ] $}}\\ 
\hline
$o_{12}$ & a line of the first kind &$2 \times 3$ %&$[q+1,q^2]$
& {\tiny{$\left [ \begin{array}{ccc} x& y & \cdot \\ \cdot & z & x \end{array}\right ] $}}\\ 
\hline
$o_{13}$ & two points &$2 \times 3$ %&$[2,q^2+q-1]$ 
& {\tiny{$\left [ \begin{array}{ccc} x& y & \cdot \\ \cdot & x & z \end{array}\right ] $}}\\ 
\hline
$o_{14}$ & three points &$2 \times 3$ %&$[3,q^2+q-2]$
& {\tiny{$\left [ \begin{array}{ccc} x& y & \cdot \\ \cdot & y & z \end{array}\right ] $}}\\ 
\hline
$o_{15}$ & one point  &$2 \times 3$ %&$[1,q^2+q]$
& {\tiny{$\left [ \begin{array}{ccc} x& ux+y & z \\ vy & x & \cdot \end{array}\right ] $}}\\ 
&(contains line of $o_{10}$)&&{\tiny{$v\lambda^2+uv\lambda - 1 \neq 0, \forall \lambda \in \F$}}\\
\hline
$o_{16}$ &  one point, &$2 \times 3$ %&$[1,q^2+q]$
& {\tiny{$\left [ \begin{array}{ccc} x& y & z \\ \cdot & x & y \end{array}\right ] $}}\\ 
&(contains line of $o_{11}$)&&\\
\hline
$o_{17}$ &$\emptyset$ &$2 \times 3$ %&$[0,q^2+q+1]$
& {\tiny{$\left [ \begin{array}{ccc} x& y & z \\ \alpha z & x+ \beta z & y+\gamma z \end{array}\right ] $}}\\ 
&&&{\tiny{$\lambda^3+\gamma \lambda^2- \beta \lambda+ \alpha \neq 0, \forall \lambda \in \F$}}\\

\hline
\end{tabular}
\caption{The eleven $\bar{H}_2$-orbits of planes in $\PG(\F^2\otimes\F^3)$  for $\F=\F_q$.}
\label{tab:planes}
\end{center}
\end{table}

\subsection{Algebraically closed fields and the real numbers}\label{subsec:alg_closed}

In this section the classification of points, lines and planes in $\PG(\F^2\otimes \F^3)$ is given for the real field and for algebraically closed fields. For other fields see \cite[Remark 3.2]{LaSh233}.

\begin{theorem}
If $\F$ is an algebraically closed field, then under the action of $\bar{H}_2$, there are 2 orbits of points, 6 orbits of lines and 9 orbits of planes in $\PG(\F^2\otimes \F^3)$. The description is as in Table \ref{tab:points}, Table \ref{tab:lines}, and Table \ref{tab:planes}, where the orbits $o_{10}$, $o_{15}$ and $o_{17}$ do not occur.
\end{theorem}
\Proof
Since $\F$ is algebraically closed, there do not exist $u,v,\alpha,\beta,\gamma \in \F$ satisfying the necessary conditions for the orbits $o_{10}$, $o_{15}$
($v\lambda^2+uv\lambda - 1 \neq 0, \forall \lambda \in \F$) and for $o_{17}$
($\lambda^3+\gamma \lambda^2- \beta \lambda+ \alpha \neq 0, \forall \lambda \in \F$).
All the other arguments used in the proof of Theorem \ref{thm:pts_lines_planes} are field independent.\qed

\begin{theorem}
If $\F$ is the field of real numbers, then there are 2 orbits of points, 7 orbits of lines and 10 orbits of planes in $\PG(\F^2\otimes \F^3)$. The description is as in Table \ref{tab:points}, Table \ref{tab:lines}, and Table \ref{tab:planes}, where the orbit $o_{17}$ does not occur.
\end{theorem}
\Proof
The orbit $o_{17}$ does not occur since a polynomial 
$\lambda^3+\gamma \lambda^2- \beta \lambda+ \alpha
\in {\mathbb{R}}[\lambda]$ always has a real root.
All the other arguments used in the proof of Theorem \ref{thm:pts_lines_planes} are valid for the real field.\qed

\section{Classification of solids and hyperplanes in $\PG(\F^2 \otimes \F^3)$}\label{sec:sols+hyps}

The classification of points, lines and planes of $\PG(\F^2 \otimes \F^3)$ allows us to classify all subspaces of $\PG(\F^2 \otimes \F^3)$ due to the following lemmas, which hold in any tensor product space $V_1\otimes V_2$,
where $V_1$, $V_2$ are two finite dimensional vectors spaces over a field $\F$. Let $K = \GL(V_1)\times \GL(V_2)$ and let $\beta_1,\beta_2$ be nondegenerate symmetric bilinear forms on $V_1$, $V_2$ respectively. Let $\beta$ be the bilinear form on $V_1 \otimes V_2$ defined by
\begin{eqnarray}
\beta(a\otimes b,c \otimes d) := \beta_1(a,c) \beta_2(b,d).
\end{eqnarray}
For any subspace $U$ of $V_1$, $V_2$ or $V_1 \otimes V_2$, we abuse notation and denote by $U^{\perp}$ the orthogonal space of $U$ with respect to $\beta_1$, $\beta_2$ or $\beta$. We define $\PG(U)^{\perp} := \PG(U^{\perp})$, and for a nonzero vector $v$ we define $v^{\perp} = \langle v\rangle ^{\perp}$. 

For $g_i \in \GL(V_i)$, denote by $\hat{g_i}$ the adjoint of $g_i$ with respect to $\beta_i$, i.e. $\beta_i(x^{g_i},y) = \beta_i(x,y^{\hat{g}_i})$ for all $x,y \in V_i$. Then the adjoint of $(g_1,g_2) \in K \leq \GL(V_1 \otimes V_2)$ with respect to $\beta$ is $(\hat{g}_1,\hat{g}_2)$, that is $\beta(x^{(g_1,g_2)},y) = \beta(x,y^{(\hat{g}_1,\hat{g}_2)})$ for all $x,y \in V_1 \otimes V_2$.

\begin{lemma}\label{lem:codim}
The $K$-orbits on $k$-spaces in $V_1\otimes V_2$ are in one-to-one correspondence with the $K$-orbits on codimension $k$-spaces in $V_1\otimes V_2$.
\end{lemma}
\Proof
Suppose $U$ is a subspace of $V_1 \otimes V_2$ of dimension $k$, and suppose $U$ is $K$-equivalent to $W$, i.e. there exist $g_1 \in \GL(V_1)$, $g_2 \in \GL(V_2)$ such that $W = U^{(g_1,g_2)} = \{x^{(g_1,g_2)} : x \in U\}$. Suppose $y \in 
W^{\perp}$, i.e. $\beta(x^{(g_1,g_2)},y)=0$ for all $x \in U$. But then $b(x,y^{(\hat{g}_1,\hat{g}_2)})=0$ for all $x \in U$, implying that $y^{(\hat{g}_1,\hat{g}_2)} \in U^{\perp}$. Hence $ (W^{\perp})^{(\hat{g}_1,\hat{g}_2)}\leq U^{\perp} $, and a consideration of dimensions concludes the proof.
\qed

\begin{lemma}\label{lem:subvarperp}
Let $U \leq \F^{n_1}$ and $V \leq \F^{n_2}$. 

(a) Then $(U \otimes V)^{\perp} = \langle U^{\perp} \otimes \F^{n_2},\F^{n_1} \otimes V^{\perp}\rangle$, and $\PG(U \otimes V)^{\perp}$ meets $S_{n_1,n_2}$ in precisely $\sigma_{n_1,n_2}(\PG(U^{\perp}) \times \PG(\F^{n_2})) \cup \sigma_{n_1,n_2}(\PG(\F^{n_1}) \times \PG(V^{\perp}))$.

(b) Suppose $\PG(A)$ is a subspace of $\PG(U \otimes V)$, but is disjoint from $\sigma_{n_1,n_2}(\PG(U) \times \PG(V))$, and $A$ has dimension $jk-\mathrm{max}\{j,k\}$, where $\dim(U)=j$ and $\dim(V)=k$. Then $\PG(A^{\perp})$ meets $S_{n_1,n_2}$ in precisely 
$\sigma_{n_1,n_2}(\PG(U^{\perp}) \times \PG(\F^{n_2})) \cup \sigma_{n_1,n_2}(\PG(\F^{n_1}) \times \PG(V^{\perp}))$.
\end{lemma}

\begin{proof}
(a) Suppose $a \otimes b \in (u \otimes v)^{\perp}$. Then $\beta_1(a,u)\beta_2(b,v)=0$, and so $a \in u^{\perp}$ or $b \in v^{\perp}$. Hence $a \otimes b \in (U \otimes V)^{\perp}$ if and only if $a \in u^{\perp}$ or $b \in v^{\perp}$ for all $u \in U,v \in V$, and hence $a \in U^{\perp}$ or $b \in V^{\perp}$.
This implies that
$\PG(U \otimes V)^{\perp}$ meets $S_{n_1,n_2}$ in precisely 
$\sigma_{n_1,n_2}(\PG(U^{\perp}) \times \PG(\F^{n_2})) \cup \sigma_{n_1,n_2}(\PG(\F^{n_1}) \times \PG(V^{\perp}))$.
and hence $ \langle U^{\perp} \otimes \F^{n_2},\F^{n_1} \otimes V^{\perp}\rangle \leq (U \otimes V)^{\perp}$. Considering dimensions, equality follows.

(b) Since $A \leq U \otimes V$, $A^{\perp} \geq (U \otimes V)^{\perp}$, and by part (a)
$\sigma_{n_1,n_2}(\PG(U^{\perp}) \times \PG(\F^{n_2})) \cup \sigma_{n_1,n_2}(\PG(\F^{n_1}) \times \PG(V^{\perp}))$
is contained in $\PG(A)^{\perp} \cap S_{n_1,n_2}$.

Suppose now $a \otimes b \in A^{\perp}$, where $a \notin U^{\perp}$ and $b \notin V^{\perp}$. Then $U \otimes V$ is not contained in $(a \otimes b)^{\perp}$, and hence $W = (U \otimes V) \cap (a \otimes b)^{\perp}$ is a hyperplane in $U \otimes V$, containing $A$. 
Now if $x \in U \cap a^{\perp}$ and $y \in V \cap b^{\perp}$, then by part (a) of this lemma, $\langle x\rangle \otimes V$ is a subspace  of dimension $k$ contained in $W$, and $U \otimes \langle y\rangle$ is a subspace of dimension $j$ contained in $W$. 
Since $\dim W=jk-1$, and $\dim A=jk-\mathrm{max}\{j,k\}$, $A$ must intersect at least the larger of the
subspaces $\langle x\rangle \otimes V$ and $U \otimes \langle y\rangle$, contradicting the hypothesis that
$\PG(A)$ is disjoint from $\sigma_{n_1,n_2}(\PG(U) \times \PG(V))$.
\end{proof}

Now using Lemma \ref{lem:codim} and Lemma \ref{lem:subvarperp}, we can classify all the $\bar{H}_2$-orbits on subspaces of $\PG(\F^2 \otimes \F^3)$, by considering $\PG(A_2^{\perp})$ for each orbit whose second contraction space $A_2$ is a point or a line. We determine these spaces and their intersection with the Segre variety $S_{2,3}$.
%, and their rank distributions $r_2(A_2)$. 

\begin{definition}
We identify $o_i$ with the $\bar{H}_2$-orbit of subspaces in $\PG(\F^2\otimes \F^3)$ containing $\PG(A_2)$, where $A\in o_i$, and
we define $o_i^{\perp}$ to be the $\bar{H}_2$-orbit of subspaces in $\PG(\F^2\otimes \F^3)$ containing $\PG(A_2)^\perp$, where $A\in o_i$.
\end{definition}
We start with the following lemma which will be used in the proof of the main result of this section. 
\begin{lemma}\label{lem:$o_12^perp$}
$o_{12}^{\perp}=o_{11}^T$.
\end{lemma}
\Proof
Let $A$ be a tensor in $o_{12}$. If $\PG(A_2^\perp)$ contains a line of the first kind of $S_{2,3}$, say $\PG(\F^2 \otimes \langle v \rangle)$, then by Lemma \ref{lem:subvarperp}, the plane $\PG(A_2)$ is contained in the $\PG(\F^2\otimes v^\perp)$, a contradiction. This implies that $\PG(A_2^\perp)$ does not contain a line of the first kind. Moreover, since $\PG(A_2)$ does contain a line of the first kind, $\PG(A_2^\perp)$ is contained in an $\langle S_{2,2}\rangle$. By Table \ref{tab:planes} this only leaves the possibility
$o_{12}^{\perp}=o_{11}^T$. 
\qed

\subsection{Finite Fields}
\begin{theorem}\label{thm:sols_hyps}
If $\F$ is a finite field, then under the action of $\bar{H}_2$, there are 7 orbits of solids, and 2 orbits of hyperplanes in $\PG(\F^2\otimes \F^3)$. The description is as in Table \ref{tab:solids} and Table \ref{tab:hyp}.
\end{theorem}
\Proof
(a) Perps of points on $S_{2,3}$: $o_1^\perp$

Consider $o_1$, where $\PG(A_2)$ is a point of $S_{2,3}$, say $x=\langle u\otimes v\rangle$. By Lemma \ref{lem:subvarperp}, $x^{\perp}$ is a projective hyperplane meeting $S_{2,3}$ precisely in the plane $\PG(u^{\perp} \otimes \F^3)$ and the Segre variety $\sigma(\F^2 \times v^{\perp})$. They intersect in the line $\PG(u^{\perp}\otimes v^{\perp})$. 

(b) Perps of lines on $S_{2,3}$: $o_2^\perp$, $(o_4^T)^\perp$
 
Next take a line $l = \PG(A_2)$ with $A \in o_2$, that is, $l$ is a line of the second kind $\PG(u \otimes \langle a,b\rangle)$ on $S_{2,3}$. Then by Lemma \ref{lem:subvarperp}, $l^{\perp}$ meets $S_{2,3}$ in the union of the plane $\PG(u^{\perp} \otimes \F^3)$ and the line $\PG(\F^2 \otimes (a^{\perp}\cap b^{\perp}))$.

Next consider a line $l = \PG(A_2)$ with $A \in o_4^T$, that is, $l = \PG(\F^2 \otimes v)$ is a line of the first kind on $S_{2,3}$. By Lemma \ref{lem:subvarperp}, $l^{\perp}$ meets $S_{2,3}$ in precisely the Segre variety $\sigma(\F^2 \times v^{\perp})$, an $S_{2,2}$.

(c) Perps of lines secant to $S_{2,3}$: $o_5^\perp$

For a secant line $l = \langle u \otimes v,w \otimes z\rangle= \PG(A_2)$ with $A \in o_5$, we get that the intersection of $l^{\perp}$ with $S_{2,3}$ is the union of the three lines $\PG(\F^2 \otimes (v^{\perp} \cap z^{\perp}))$, $\PG(u^{\perp} \otimes z^{\perp})$ and $\PG(w^{\perp} \otimes v^{\perp})$. The second and third line are disjoint, and both meet the first line. 

(d) Perps of lines tangent to $S_{2,3}$: $o_6^\perp$, $o_7^\perp$

Suppose $l$ is a tangent line to $S_{2,3}$ at a point $\PG(u \otimes v)$. Let $y$ be a point of rank two on $l$. The intersection of $l^{\perp}$ with $S_{2,3}$ is then equal to the union of the intersections $y^{\perp} \cap \PG(u^{\perp} \otimes \F^3)$ and $y^{\perp} \cap \PG(\F^2 \otimes v^{\perp})$,

Then $l^{\perp}$ contains a line of the first kind, say $\PG(\F^2 \otimes b)$, if and only if $l$ is contained in $\PG(\F^2 \otimes b)^{\perp}$, which is equal to $\PG(\F^2 \otimes b^{\perp})$, and so $l$ is spanned by an $S_{2,2}$. Hence $l^{\perp}$ contains a line of the first kind if and only if it is the second contraction space of a tensor in $o_6^{\perp}$, and this line is unique.

Now $l^{\perp}$ contains a line of the second kind, say $\PG(a \otimes b^{\perp})$, if and only if $l$ is contained in $\PG(a \otimes b^{\perp})^{\perp}$, which equals $\langle a^{\perp} \otimes \F^3,\F^2 \otimes b\rangle$. But this contains a plane and a line of $S_{2,3}$ meeting in a point, and hence this point must be $\PG(u \otimes v)$, for otherwise $l$ would be a secant to $S_{2,3}$. Hence $a= u^{\perp}$ and $v \in b^{\perp}$, and this line is unique, for otherwise $l^{\perp}$ would contain $\PG(u^{\perp} \otimes \F^3)$, and $l$ would be contained in $\PG(u \otimes \F^3)$.

Hence the second contraction space of a tensor in $o_6^{\perp}$ meets $S_{2,3}$ in precisely two lines, one of each kind.

Now the second contraction space of a tensor in $o_7^{\perp}$ meets $\PG(\F^2 \otimes v^{\perp})$ in a plane. By the above, this plane cannot contain a line of $S_{2,3}$, and hence this contraction space meets $S_{2,3}$ in a conic and a line of the second kind.

(e) Perps of lines disjoint from $S_{2,3}$: $o_{10}^\perp$, $o_{11}^\perp$

If $l = \PG(A_2)$ with $A \in o_{10}$, i.e. $l$ is contained in some $\langle S_{2,2}\rangle$ but disjoint from $S_{2,2}$. Let $l \subset \langle \sigma(\PG(\F^2)\times m)\rangle$, where $m$ is a line of $\PG(\F^3)$.
Then, by Lemma \ref{lem:subvarperp}(b), $l^{\perp}$, the second contraction space of a tensor in $o_{10}^{\perp}$, meets $S_{2,3}$ in precisely the line $\sigma(\PG(\F^2) \times m^{\perp})$.

The second contraction space of $o_{17}$ is a plane disjoint from $S_{2,3}$, and hence any line contained in it is the second contraction space of a tensor in $o_{10}$ or $o_{11}$. We claim that $o_{10}$ is not possible. For suppose otherwise, i.e. $l$ is a line disjoint from $S_{2,3}$, contained in an $S_{2,2}$ (say, $\sigma(m_1 \times m_2$)), and contained in a plane $\pi$ disjoint from $S_{2.3}$. Now the line $m = \PG(\F^2 \otimes m_2^{\perp})$ is contained in $l^{\perp}$. But as $\pi^{\perp}$ is contained in $l^{\perp}$, we must have that $m$ meets $\pi^{\perp}$. But $m$ is contained in $S_{2,3}$, and $\pi^{\perp}$ is disjoint from $S_{2,3}$ by Lemma \ref{lem:subvarperp}, a contradiction.

Hence the second contraction space $l$ of a tensor in $o_{11}$ is contained in a plane $\pi$ disjoint from $S_{2,3}$, and so $l^{\perp}$, the second contraction space of a tensor in $o_{11}^{\perp}$, also contains a plane $\pi^{\perp}$, which is disjoint from $S_{2,3}$ by Lemma \ref{lem:subvarperp} (b). 

We may view the planes contained in this Segre variety as the image of the points on a subline $b\cong \PG(1,q)$ of the projective line $\PG(1,q^3)$ after applying the field reduction map ${\mathcal{F}}:={\mathcal{F}}_{2,3,q}$ (see \cite{LaVaPrep}). Since there is one orbit of planes disjoint from $S_{2,3}$, we may take the plane $\pi$ to be the image ${\mathcal{F}}(\Theta)$ of a point $\Theta$ of $\PG(1,q^n)\setminus b$ under $\mathcal F_{2,3,q}$. By \cite[Theorem 3.3]{LaZaPrep} each solid containing ${\mathcal{F}}(\Theta)$ intersects $S_{2,3}$ in a normal rational curve of degree 3 for $q\geq 3$ and of degree 2 for $q=2$.

(f) Perps of points not on $S_{2,3}$: $o_4^\perp$

Let $x$ be a point of rank two. As we have seen above, the perp of a point on $S_{2,3}$ contains an $S_{2,2}$. Viceversa, every hyperplane $\pi$ containing an $S_{2,2}$ is the perp of a point on $S_{2,3}$. For suppose $\pi$ contains 
$\sigma_{2,3}(\PG(\F^2) \times \ell)$, for some line $\ell \in \PG(\F^3)$. Then each plane 
$\langle u\otimes \F^3\rangle$ contained in $S_{2,3}$ is contained in a unique hyperplane containing 
$\sigma_{2,3}(\PG(\F^2) \times \ell)$, and no two such hyperplanes contain more than one of these planes.
It follows that the perp $x^\perp$ of the rank two point $x$ does not contain an $S_{2,2}$. But $x^\perp$ does contain
a unique line $\ell_x$ of type $o_4^T$, since we have seen that the perp of such a line is a solid containing an $S_{2,2}$ and,  by \cite[Lemma 2.4]{LaSh233} $x$ is contained in a unique such solid $\langle Q(x)\rangle$. For each point $y\in \ell_x$, the unique plane $\rho_y\ni y$ contained in $S_{2,3}$, intersects $x^\perp$ in a line $r_y$. The intersection of $x^\perp$ with $S_{2,3}$ is the union of these lines $\{r_y~:~y \in \ell_x\}$.
Another way to describe this hyperplane is by considering a plane $\tau$ on $x$ meeting $S_{2,3}$ in a line $t$ of the first kind, i.e. $\tau$ is a plane of type $o_{12}$ . 
Then, by Lemma \ref{lem:$o_12^perp$},
$\tau^\perp$ is a plane of type $o_{11}^T$ meeting $S_{2,3}$ in the conic $\mathcal C$ which is the intersection of $t^\perp$ with $x^\perp$, and $\tau^\perp$ is disjoint from $\ell_x$. Each solid $\kappa$ on $t$ intersecting $S_{2,3}$ in an $S_{2,2}$ gives a line $\kappa^\perp$ of the first kind intersecting $\tau^\perp$ in a point of $\mathcal C$. The lines $r_y$, $y\in \ell_x$ give a one-to-one correspondence between the points of $\ell_x$ and the points of $\mathcal C$.  This is known as a cubic scroll in $\PG(4,\F)$.
\qed

\medskip

%\subsection{Summary}
We summarise the orbits of solids and hyperplanes in the following tables. We also include the number of rank one points contained in a representative of the orbit.

\begin{table}[H]
\begin{center}
\begin{tabular}{|c|c|c|c|c|}
\hline
{\bf Orbit} & {\bf Intersection with} $\mathbf{S_{2,3}}$ &{\bf Min}& {\bf Rep}
&{\bf $\#$ rank 1}\\
\hline
$o_2^{\perp}$ & plane and line & $2\times 3$
& {\tiny{$\left [ \begin{array}{ccc} \cdot & \cdot & x \\ y& z & w \end{array}\right ] $}}
&$q^2+2q+1$ \\ 
\hline
$(o_4^T)^{\perp}$ & $S_{2,2}$ & $2\times 2$
& {\tiny{$\left [ \begin{array}{ccc} \cdot & x & y \\ \cdot& z & w \end{array}\right ] $}}
&$q^2+2q+1$ \\ 
\hline
$o_5^{\perp}$ & three lines & $2\times 3$
& {\tiny{$\left [ \begin{array}{ccc} \cdot & x & y \\ z& \cdot&  w \end{array}\right ] $}}
&$3q+1$ \\ 
\hline
$o_6^{\perp}$ & two intersecting lines & $2\times 3$
& {\tiny{$\left [ \begin{array}{ccc} y & y & z \\ \cdot& -x & w \end{array}\right ] $}}
&$2q+1$ \\ 
\hline
$o_7^{\perp}$ & conic and line & $2\times 3$
& {\tiny{$\left [ \begin{array}{ccc} x & \cdot & y \\ z& w & -x \end{array}\right ] $}}
&$2q+1$ \\ 
\hline
$o_{10}^{\perp}$ &line & $2\times 3$
& {\tiny{$\left [ \begin{array}{ccc} x & -vy & z \\ y& -x+uvy & w \end{array}\right ] $}}
&$q+1$\\ 
 &&&{\tiny{$v\lambda^2+uv\lambda - 1 \neq 0, \forall \lambda \in \F$}}& \\
\hline
$o_{11}^{\perp}$ &normal rational curve & $2\times 3$
& {\tiny{$\left [ \begin{array}{ccc}  x & y & z \\ w& -x & -y \end{array}\right ] $}}
&$q+1$\\ 
\hline
\end{tabular}
\caption{The seven orbits of solids in $\PG(\F^2\otimes\F^3)$ for $\F=\F_q$.}
\label{tab:solids}
\end{center}
\end{table}

\begin{table}[h]
\begin{center}
\begin{tabular}{|c|c|c|c|c|}
\hline
{\bf Orbit} & {\bf Intersection with} $\mathbf{S_{2,3}}$ & {\bf Min} & {\bf Rep}
&{\bf $\#$ rank 1}\\
\hline
$o_1^{\perp}$ & plane and $S_{2,2}$ & $2\times 3$
& {\tiny{$\left [ \begin{array}{ccc}  \cdot & x & y \\ z& w & t \end{array}\right ] $}}
&$2q^2+2q+1$ \\
\hline
$o_4^{\perp}$ & cubic scroll  & $2\times 3$
& {\tiny{$\left [ \begin{array}{ccc}  x & y & z \\ w& -x & t \end{array}\right ] $}}
&$(q+1)^2$ \\ 
\hline
\end{tabular}
\caption{The two orbits of hyperplanes in $\PG(\F^2\otimes\F^3)$ for $\F=\F_q$.}
\label{tab:hyp}
\end{center}
\end{table}

For the sake of completeness we include the perps of the $\bar{H}_2$-orbits of planes in $\PG(\F^2\otimes \F^3)$ as listed in Table \ref{tab:planes}.
\begin{theorem}
For each $\bar{H}_2$-orbit $o_i$, whose representative $\PG(A_2)$ is a plane, we have that $o_i = o_i^{\perp}$, with the exception $o_{12}^{\perp}=o_{11}^T$.
\end{theorem}
\Proof
The exception is Lemma \ref{lem:$o_12^perp$}. The other cases are easily obtained using Lemma \ref{lem:subvarperp} and Theorem \ref{thm:sols_hyps}.
\qed

\subsection{Algebraically closed fields and the real numbers}
The classification of solids and hyperplanes in $\PG(\F^2\otimes \F^3)$ for the real field and for algebraically closed fields easily follows from the finite field case together with Section \ref{subsec:alg_closed}.
\begin{theorem}\label{thm:sols_hyps_alg_closed}
If $\F$ is an algebraically closed field, then under the action of $\bar{H}_2$, there are 6 orbits of solids, and 2 orbits of hyperplanes in $\PG(\F^2\otimes \F^3)$. The description is as in Table \ref{tab:solids} and Table \ref{tab:hyp} where the orbit $o_{10}^\perp$ does not occur.
\end{theorem}

\begin{theorem}\label{thm:sols_hyps_real}
If $\F$ is the field of real numbers, then under the action of $\bar{H}_2$, there are 7 orbits of solids, and 2 orbits of hyperplanes in $\PG(\F^2\otimes \F^3)$. The description is as in Table \ref{tab:solids} and Table \ref{tab:hyp}.
\end{theorem}

\section{Orbits of tensors in $\F^2\otimes \F^3 \otimes \F^r$}\label{sec:summary_r}
Now consider the more general case of a tensor product with three factors where the last factor has arbitrary finite dimension $r\geq 1$, i.e. $\F^2\otimes \F^3 \otimes \F^r$. 
As before denote by $G$ and $H$ the natural groups acting on $\F^2\otimes \F^3 \otimes \F^r$ and by $H_i$, $i=1,2,3$, the induced groups in the contraction spaces.

Combining Lemma \ref{lem:codim} and \cite[Lemma 2.1]{LaSh233}  with the previous section and the following lemma, we can classify $H$-orbits on tensors in $\F^2 \otimes \F^3 \otimes \F^r$ for all positive integers $r$.
\begin{lemma}\label{lem:orbits_equivalence}
For all integers $r\geq 6$, the number of $H$-orbits on tensors in $\F^2 \otimes \F^3 \otimes \F^r$ is equal to the number of $H$-orbits on tensors in $\F^2 \otimes \F^3 \otimes \F^6$, which in turn equals the number of $H_2$-orbits of subspaces of $\F^2 \otimes \F^3$.
\end{lemma}
\Proof
Let $A \in \F^2 \otimes \F^3 \otimes \F^r$, and consider its third contraction space $A_3 \leq \F^2 \otimes \F^3$. Then $\dim(A_3) = k\leq 6$. Let $a_1,\ldots,a_k$ be a basis for $A_3$, and let $e_1,\ldots,e_r$ be a basis for $\F^r$. Define $B = \sum_{i=1}^k a_i \otimes e_i \in  \F^2 \otimes \F^3 \otimes \F^r$. Then the third contraction space $B_3$ of $B$ is equal to $A_3$, and hence by \cite[Lemma 2.1]{LaSh233}, $A$ is 
$H$-equivalent to $B$, and the result follows.
\qed

Hence from the previous sections, and from \cite{LaSh233}, we have the following theorem. A computer classification for the cases $r=2,3,4$, $\F=\F_2$ was done in \cite{BrHu2012}.

\begin{theorem}
The number of $H$-orbits of tensors in $\Fq^2 \otimes \Fq^3 \otimes \Fq^r$ is as listed in the following table:
\[
\begin{array}{|c||c|c|c|c|c|c|}
\hline
r&1&2&3&4&5&\geq6\\
\hline
$\#$ $H$-orbits&3&10&21&28&30&31\\
\hline
\end{array}
\]

\end{theorem}
\Proof
First consider the case $r\geq 6$. It follows from Theorem \ref{thm:pts_lines_planes} and Theorem \ref{thm:sols_hyps} that there are 2 $\bar{H}_2$-orbits on points, 7 on lines, 11 on planes, 7 on solids, and 2 $\bar{H}_2$-orbits on hyperplanes of $\PG(\F_q^2\otimes \F_q^3)$. This amounts to 29 orbits. Including the trivial subspaces $0$ and $\F_q^2\otimes \F_q^3$ one obtains 31 $H_2$-orbits of subspaces of $\F_q^2\otimes \F_q^3$, which by Lemma \ref{lem:orbits_equivalence}, equals the number of $H$-orbits of tensors in $\Fq^2 \otimes \Fq^3 \otimes \Fq^r$.
If $r\leq 6$, then not all these orbits occur, since
the third contraction space $A_3$ of a tensor $A$ will have dimension at most $r$ in $\F_q^2\otimes \F_q^3$. This gives $30$ orbits for $r=5$; $28$ orbits for $r=4$; $21$ orbits for $r=3$; and $10$ orbits for $r=2$.
\qed

\begin{theorem}
The number of $G$-orbits of tensors in $\Fq^2 \otimes \Fq^3 \otimes \Fq^r$ is as listed in the following table:
\[
\begin{array}{|c||c|c|c|c|}
\hline
r&1&2&3&\geq4\\
\hline
$\#$ $G$-orbits &3&9&18&$\#$ $H$-orbits\\
\hline
\end{array}
\]

\end{theorem}
\Proof
It follows from the definition of $G$ that $G=H$ unless $r \in \{2,3\}$. When $r=2$, the $H$-orbits are $o_0,o_1,o_2,o_4,o_4^T,o_5,o_6,o_7,o_{10},o_{11}$, and the orbits $o_2$ and $o_4$ are $G$-equivalent. 
When $r=3$ the result is part of Theorem \ref{thm:233}.
\qed

For the sake of completeness we include the corresponding tables with the number of orbits in for algebraically closed fields and for the field of real numbers. The proof is a straightforward consequence from the finite field case and the previous sections.
\begin{theorem}
If $\F$ is an algebraically closed field or the field of real numbers, then the number of $H$-orbits and $G$-orbits of tensors in $\F^2 \otimes \F^3 \otimes \F^r$ is as listed in the following tables.
\[
\begin{array}{|c||c|c|c|c|c|c|c|}
\hline
r&1&2&3&4&5&\geq 6 &\\
\hline
$\#$ $H$-orbits&3&9&18&24&26&27& \F ~\mbox{algebraically closed}\\
\hline
$\#$ $H$-orbits&3&10&20&27&29&30& \F={\mathbb{R}}\\
\hline
\end{array}
\]
\[
\begin{array}{|c||c|c|c|c|c|}
\hline
r&1&2&3&\geq4 &\\
\hline
$\#$ $G$-orbits &3&8&15&$\#$ $H$-orbits & \F~\mbox{algebraically closed}\\
\hline
$\#$ $G$-orbits &3&9&17&$\#$ $H$-orbits & \F={\mathbb{R}}\\
\hline
\end{array}
\]
\end{theorem}

\section{Tensor rank in $\F^2\otimes \F^3 \otimes \F^r$}\label{sec:rank}
In this section we determine the rank of the tensors in $\F^2\otimes \F^3 \otimes \F^r$. Since the rank of a tensor is $G$- and $H$-invariant it makes sense to define the {\it tensor rank} of a $G$- or $H$-orbit $o$ as the rank of a tensor $A \in o$. For ease of comparison with the previous sections, the following theorem is stated in terms of the $H$-orbits.
\begin{theorem}
The tensor rank of the $H$-orbits of tensors in $\F^2\otimes \F^2 \otimes \F^r$, $\F$ finite and $|\F| > 2$, are as listed in Table \ref{tab:trk}. When $\F$ is real or algebraically closed, the same holds for all non-empty orbits. When $\F=\F_2$, the same holds for all orbits with the exception that $o_{17}$ and $o_{11}^{\perp}$ have tensor rank $5$. 
\end{theorem}
\begin{table}[H]
\begin{center}
\begin{tabular}{|c|l|}
\hline
\textbf{Trk} & \textbf{Orbits}\\
\hline
\hline
0 & $o_0$\\
\hline
1 & $o_1$\\
\hline
2& $o_2,o_4,o_4^T,o_5$\\
\hline
3&$o_3,o_6,o_7,o_7^T,o_8,o_{10},o_{11},o_{11}^T,o_{14}$\\
\hline
4 & $o_9,o_{12},o_{13},o_{15},o_{16},o_{17}, o_2^{\perp},(o_4^T)^{\perp},o_5^{\perp},o_7^{\perp},o_{11}^{\perp}$\\
\hline
5&$o_1^{\perp},o_4^{\perp},o_6^{\perp},o_{10}^{\perp}$\\
\hline
6&$o_0^{\perp}$\\
\hline
\end{tabular}
\caption{Tensor ranks of orbits in $\F^2\otimes\F^3 \otimes \F^r$, $|\F| > 2$.}
\label{tab:trk}
\end{center}
\end{table}
\proof
The rank of a tensor equals the rank of the contraction spaces (see e.g. \cite[Proposition 2.1]{LaSh201?}). This implies that the rank of a tensor is at least the maximum of the dimensions of its contraction spaces. The rank of a tensor is equal to this maximum dimension if and only if its (projective) contraction space is spanned by points on the appropriate Segre variety.

From Table \ref{tab:hyp} we see that each hyperplane in $\PG(\F^2 \otimes \F^3)$ is spanned by five points on the Segre variety $S_{2,3}$, and hence the representatives of both the orbits $o_1^{\perp}$ and $o_4^{\perp}$ have tensor rank $5$. Moreover, it follows that every orbit excluding $o_0^{\perp}$ (which has tensor rank $6$) has tensor rank at most $5$.

If the first contraction space gives a tangent to the Segre variety, then the rank of the corresponding tensor is at least three. Similarly, if the second or third contraction space gives a plane, but the intersection with the Segre variety does not span that plane, then the rank of the corresponding tensor is at least four.
In combination with the explicit descriptions in the table, this determines the rank of the orbits $o_1, \ldots, o_9, o_{12}$, and $o_{13}$.

The orbit $o_{10}$ has tensor rank three, since the second contraction space is contained in the subspace generated by $(e_1+e_2)\otimes(e_1+e_2)$, $e_1\otimes e_2$ and $e_2\otimes e_1$. The orbit $o_{11}$ has rank three since the third contraction space is a plane meeting the Segre variety in a conic. The orbit $o_{14}$ has rank 3 since the three points of rank one in the second contraction space are not collinear. The orbit $o_{15}$ has rank 4, since the second contraction space is spanned by $(e_1+e_2)\otimes(e_1+e_2)$, $e_1\otimes e_2$, $e_2\otimes e_1$, and 
$e_1\otimes e_3$. Similarly, the second contraction space of the representative of the orbit $o_{16}$ is spanned by
$e_1\otimes e_1$, $e_1\otimes e_3$, $e_2\otimes e_2$, and $(e_1+e_2)\otimes (e_2+e_3)$. 

We can see from Table \ref{tab:solids} that each of the solids corresponding to the second contraction spaces of tensors in $o_2^{\perp},(o_4^T)^{\perp},o_5^{\perp}$, and $o_7^{\perp}$ are spanned by points of $S_{2,3}$, and hence each of these orbits has tensor rank $4$. The second contraction space of a tensor in $o_{11}^{\perp}$ contains a normal rational curve of degree 3 for $q\geq 3$ and of degree 2 for $q=2$. It follows that the orbit $o_{11}^{\perp}$ has tensor rank $4$ for $q\geq 3$. One easily verifies that the rank is 5 for $q=2$.

Suppose $A$ belongs to the orbit $o_{17}$, which must have tensor rank at least $4$, as its second contraction space is not spanned by points of $S_{2,3}$ (indeed, it is disjoint from $S_{2,3}$). Then $\PG(A_2)$ is contained in a solid which is the second contraction space of a tensor in $o_{11}^{\perp}$, and so by the previous paragraph $o_{17}$ has tensor rank $4$ for $q \geq 3$ and tensor rank $5$ for $q=2$.

The second contraction spaces of representatives of the orbits $o_6^{\perp}$ and $o_{10}^{\perp}$ are not spanned by points of $S_{2,3}$, and so must have tensor rank $5$.
\qed

\begin{remark}
Our original motivation for the study of the tensor rank comes from their relation with semifields, i.e. finite non-associative division algebras: to each semifield one can associate a tensor whose rank is an isotopism invariant, see \cite{Lavrauw2012}.  In this context the first case to consider is $\F^3\otimes \F^3 \otimes \F^3$, which will be the subject of future work and for which \cite{LaPaZa2013} and the present paper will be of great value.
\end{remark}

\section{Summary}

Here we collect information about each $G$-orbit in $\F^2\otimes \F^2 \otimes \F^3$ and their contraction spaces for the finite field $\F_q$. If $\PG(A_3)$ is not listed, then it is equivalent to $\PG(A_2)$. The third column contains the tensor rank and the rank distributions of the projective contraction spaces, which we write as multisets for convenience.

\bigskip

\begin{tabular}{|l|l|l|}
\hline
  & {\bf Description and intersection with $S_{2,3}$} & {\bf Trk and Rank Distr.} \\
\hline
 $o_1$& $ e_1 \otimes e_1 \otimes e_1 $& 1\\
$\PG(A_1)$:& Point on $S_{3,3}$& $\{1^1\}$\\
$\PG(A_2)$:& Point on $S_{2,3}$&  $\{1^1\}$\\
\hline
$o_2$& $ e_1 \otimes (e_1 \otimes e_1+ e_2\otimes e_2) $& 2\\
$\PG(A_1)$:& Point of rank two& $\{2^1\}$\\
$\PG(A_2)$:& Line on $S_{2,3}$&  $\{1^{q+1}\}$\\
\hline
$o_3$& $ e_1 \otimes (e_1 \otimes e_1+ e_2\otimes e_2 + e_3 \otimes e_3) $& 3\\
$\PG(A_1)$:& Point of rank three& $\{3^1\}$\\
$\PG(A_2)$:& Plane on $S_{2,3}$&  $\{1^{q^2+q+1}\}$\\
\hline
$o_4$& $ e_1 \otimes e_1 \otimes e_1+ e_2\otimes e_1 \otimes e_2  $&2\\
$\PG(A_1)$:& Line on $S_{3,3}$& $\{1^{q+1}\}$\\
$\PG(A_2)$:& Point of rank two&  $\{2^1\}$\\
$\PG(A_3)$:& Line on $S_{2,3}$& $\{1^{q+1}\}$\\
\hline
 $o_5$& $ e_1 \otimes e_1 \otimes e_1+ e_2\otimes e_2 \otimes e_2  $&2\\
$\PG(A_1)$:& Secant line& $\{1^2,2^{q-1}\}$\\
$\PG(A_2)$:& Secant line&   $\{1^2,2^{q-1}\}$\\
\hline
 $o_6$ & $ e_1 \otimes e_1 \otimes e_1+ e_2\otimes (e_1 \otimes e_2 + e_2 \otimes e_1)  $&3\\
$\PG(A_1)$& Tangent line, contained in an $\langle S_{2,2}\rangle$& $\{1,2^{q}\}$\\
$\PG(A_2)$& Tangent line, contained in an $\langle S_{2,2}\rangle$&   $\{1,2^{q}\}$\\
\hline
 $o_7$& $ e_1 \otimes e_1 \otimes e_3+ e_2\otimes (e_1 \otimes e_1 + e_2 \otimes e_2)  $&3\\
$\PG(A_1)$& Tangent line, contained in an $\langle S_{2,3}\rangle$, & $\{1,2^{q}\}$\\
&not contained in an $\langle S_{2,2}\rangle$&\\
$\PG(A_2)$& Tangent line, not contained in an $\langle S_{2,2}\rangle$&   $\{1,2^{q}\}$\\
$\PG(A_3)$& Plane containing two lines of an $S_{2,2}$, &   $\{1^{2q+1},2^{q^2-q}\}$\\
\hline
 $o_8$& $ e_1 \otimes e_1 \otimes e_1+ e_2\otimes (e_2 \otimes e_2 + e_3 \otimes e_3)  $&3\\
$\PG(A_1)$& Tangent line, not contained in an $\langle S_{2,3}\rangle$, & $\{1,2,3^{q-1}\}$\\
&containing a point of rank two&\\
$\PG(A_2)$& Plane, containing a line and a point of $S_{2,3}$, &   $\{1^{q+2},2^{q^2-1}\}$\\
&not contained in an $\langle S_{2,2}\rangle$&\\
\hline
\end{tabular}

\begin{tabular}{|l|l|l|}
\hline
$o_9$& $ e_1 \otimes e_3 \otimes e_1+ e_2\otimes (e_1 \otimes e_1 + e_2 \otimes e_2 + e_3 \otimes e_3)  $& 4\\
$\PG(A_1)$& Tangent line, not contained in an $\langle S_{2,3}\rangle$, & $\{1,3^{q}\}$\\
&not containing a point of rank two&\\
$\PG(A_2)$& Plane, containing a line of $S_{2,3}$, &   $\{1^{q+1},2^{q^2}\}$\\
&not contained in an $\langle S_{2,2}\rangle$&\\
\hline
$o_{10}$& $ e_1\otimes (e_1\otimes e_1+ e_2\otimes e_2+u e_1\otimes e_2) +  
e_2\otimes (e_1\otimes e_2+v e_2\otimes e_1),$ & 3\\
 & $v\lambda^2+uv\lambda - 1 \neq 0$ for all $\lambda \in \F$& \\
$\PG(A_1)$& Line, constant rank two, contained in an $\langle S_{2,2}\rangle$& $\{2^{q+1}\}$\\
$\PG(A_2)$& Line, constant rank two, contained in an $\langle S_{2,2}\rangle$& $\{2^{q+1}\}$\\
\hline
$o_{11}$& $ e_1\otimes (e_1 \otimes e_1 + e_2 \otimes e_2)+ e_2\otimes (e_1 \otimes e_2 + e_2 \otimes e_3)$& 3\\
$\PG(A_1)$& Line, constant rank two, contained in an $\langle S_{2,3}\rangle$ & $\{2^{q+1}\}$ \\
 & but not in an $\langle S_{2,2}\rangle$&  \\
$\PG(A_2)$& Line, constant rank two, contained in an $\langle S_{2,3}\rangle$ & $\{2^{q+1}\}$\\
 & but not in an $\langle S_{2,2}\rangle$&  \\
$\PG(A_3)$& Plane in an $\langle S_{2,2}\rangle$, meeting in a conic & $\{1^{q+1},2^{q^2}\}$\\
\hline
$o_{12}$& $ e_1\otimes (e_1 \otimes e_1 + e_2 \otimes e_2)+ e_2\otimes (e_1 \otimes e_3 + e_3 \otimes e_2)$& 4\\
$\PG(A_1)$& Line, constant rank two, not contained in an $\langle S_{2,3}\rangle$& $\{2^{q+1}\}$\\
$\PG(A_2)$& Plane containing a line of $S_{2,3}$ & $\{1^{q+1},2^{q^2}\}$\\
\hline
$o_{13}$& $ e_1\otimes (e_1 \otimes e_1 + e_2 \otimes e_2)+ e_2\otimes (e_1 \otimes e_2 + e_3 \otimes e_3)$& 4\\
$\PG(A_1)$& Line, two points of rank $2$ & $\{2^2,3^{q-1}\}$\\
$\PG(A_2)$& Plane containing two points of $S_{2,3}$& $\{1^{2},2^{q^2+q-1}\}$\\
\hline
$o_{14}$& $ e_1\otimes (e_1 \otimes e_1 + e_2 \otimes e_2)+ e_2\otimes (e_2 \otimes e_2 + e_3 \otimes e_3)$& 3\\
$\PG(A_1)$& Line, three points of rank $2$ & $\{2^3,3^{q-2}\}$\\
$\PG(A_2)$& Plane containing three points of $S_{2,3}$ & $\{1^{3},2^{q^2+q-2}\}$\\
\hline
$o_{15}$& 
$ e_1\otimes (e+u e_1\otimes e_2) + e_2\otimes (e_1\otimes e_2+v e_2\otimes e_1),$& 4\\
 & $v\lambda^2+uv\lambda - 1 \neq 0$ for all
$\lambda \in \F$& \\
$\PG(A_1)$& Line, one point of rank $2$ & $\{2,3^{q}\}$\\
$\PG(A_2)$& Plane containing one point of $S_{2,3}$ & $\{1,2^{q^2+q}\}$\\
\hline
$o_{16}$& $ e_1\otimes (e_1 \otimes e_1 + e_2 \otimes e_2+e_3 \otimes e_3)+ e_2\otimes (e_1 \otimes e_2 + e_2 \otimes e_3)$& 4\\
$\PG(A_1)$& Line, one point of rank $2$ & $\{2,3^{q}\}$\\
$\PG(A_2)$& Plane containing one point of $S_{2,3}$ & $\{1,2^{q^2+q}\}$\\
\hline
$o_{17}$& $ e_1\otimes e+ 
e_2\otimes (e_1\otimes e_2 + e_2\otimes  e_3 + e_3\otimes (\alpha e_1 + \beta e_2 + \gamma e_3)),$ & 4 if $q\geq 3$\\
 & $\lambda^3+\gamma \lambda^2- \beta \lambda+ \alpha \neq 0$ for all $\lambda \in \F$ & 5 if $q=2$\\

$\PG(A_1)$& Line, constant rank $3$ & $\{3^{q+1}\}$\\
$\PG(A_2)$& Plane disjoint from $S_{2,3}$ & $\{2^{q^2+q+1}\}$\\
\hline
\end{tabular}

\vspace{1 cm}

{\bf Acknowledgement}
The authors would like to thank the members of the research group Algebra and Number Theory of Sabanci University for their kindness and hospitality, thanks to which the final part of this research was completed. Geometric insights and initial data was obtained with the help of the GAP-package FinInG (\cite{GAP}, \cite{FinInG}).

%%%%%%%%%%%%%%%%%%%%%%%%%%%%%%%%%%%%%%%%%%%%%%%%%%%%%%%%%%%%%%%%%%%%%%%%%%%%%
%%  BIBLIOGRAPHY 
%%%%%%%%%%%%%%%%%%%%%%%%%%%%%%%%%%%%%%%%%%%%%%%%%%%%%%%%%%%%%%%%%%%%%%%%%%%%%


\begin{thebibliography}{9}


\bibitem{Blaser2004}
M. Bl\"aser. A complete characterization of the algebras of minimal bilinear complexity. {\em SIAM J. Comput.} 34 (2004/05) 277-298.




\bibitem{BrHu2012}
M.R. Bremner; J. Hu. Canonical forms of small tensors over $\F_2$, arxiv.org/abs/1206.5179v1.

\bibitem{BrSt2013}
M.R. Bremner; S.G. Stavrou. Canonical forms of $2\! \times\! 2\! \times \! 2$ and $2\! \times\! 2\! \times \! 2\! \times\!  2$ arrays over $\F_2$ and $\F_3$, {\em Linear and Multilinear Algebra}, 61 (2013), 986-997.
\bibitem{BuClSh1997}
P. Burgisser; M. Clausen; M. A. Shokrollahi, Algebraic Complexity Theory, Springer, New York, 1997.
\bibitem{GAP}
The GAP Group, GAP -- Groups, Algorithms, and Programming, Version 4.6.2; 2013. (http://www.gap-system.org).
\bibitem{FinInG}
FinInG - a GAP package for Finite Incidence Geometry, 1.01 2014, (http://cage.ugent.be/geometry/fining) J. Bamberg; A. Betten; P. Cara; J. De Beule; M. Lavrauw; M. Neunhoeffer.

\bibitem{Gantmacher}
F. R. Gantmacher. {\it The theory of matrices}. Vols. 1, 2, Translated by K. A. Hirsch, Chelsea Publishing Co., New York, 1959.

\bibitem{GlGuMaGu2006}
D.G. Glynn, T.A. Gulliver, J.G. Maks, M.K. Gupta. 
The Geometry of Additive Quantum Codes. Available from {\url www.maths.adelaide.edu.au/rey.casse/DavidGlynn/QMonoDraft.pdf}.

\bibitem{Gurvits04}
L. Gurvits. Classical complexity and quantum entanglement, {\em J. Comput. System Sci.} 69 (2004), 448-484. 

\bibitem{HaOdSa2012}
H. Havlicek; B. Odehnal; M. Saniga. On invariant notions of Segre varieties in binary projective spaces. {\em Des. Codes Cryptogr.} 62  (2012) 343-356. 

\bibitem{Heydari2008} H. Heydari. Geometrical structure of entangled states and the secant variety, {\em Quantum Inf. Process.}, 7 (2008) 43-50.

\bibitem{JaJa1979}
J. Ja'Ja'. Optimal evaluation of pairs of bilinear forms, {\em SIAM J. Comput.} 8 (1979), 443-462.

\bibitem{Landsberg2012} J.M. Landsberg. Tensors: Geometry and Applications. 2012. Graduate Studies in Mathematics, 128. American Mathematical Society, Providence, RI, 2012. xx+439 pp. ISBN: 978-0-8218-6907-9.
\bibitem{LaQiYe2012}
J. M. Landsberg; Yang Qi; Ke Ye. On the geometry of tensor network states, {\em Quantum Information \& Computation} 12 (2012) 346-354.

\bibitem{Lavrauw2012} M. Lavrauw. Finite semifields and nonsingular tensors. {\em Des. Codes Cryptogr.} 68 (2013) 205-227.
\bibitem{LaPaZa2013}
M. Lavrauw, A. Pavan and C. Zanella. On the rank of $3\times 3\times 3$-tensors. {\em Linear and Multilinear Algebra} 61 (2013) 648--652.
\bibitem{LaSh2014} M. Lavrauw and J. Sheekey. Orbits of the stabiliser group of the Segre variety product of three projective lines. {\em Finite Fields Appl.} 26 (2014) 1--6.
\bibitem{LaSh201?} M. Lavrauw and J. Sheekey. Aspects of tensor products over finite
fields and Galois geometries. Proceedings of the  Academy Contact Forum 
{\em Galois Geometries and Applications} at the {\em Royal Flemish 
Academy of   Belgium for Science and the Arts}  (October 5,  2012),
(2014), S. Nikova, B. Preneel and L. Storme, Eds., 95--102.
\bibitem{LaSh233} M. Lavrauw and J. Sheekey. Canonical forms of $2 \times 3 \times 3$ tensors over the real field, algebraically closed fields, and finite fields. To appear in {\em Linear Algebra and its Applications}.
\bibitem{LaVaPrep}
M. Lavrauw and G. Van de Voorde. Field reduction and linear sets in finite geometry. In Topics in Finite Fields (eds. Kyureghyan et al.) {\em AMS Contemporary Math.} Volume {\bf 632} 2015, 271--293. 
\bibitem{LaZaPrep}
M. Lavrauw and C. Zanella. Subspaces intersecting each element of a regulus in one point, Andr\'e-Bruck-Bose representation and clubs. arxiv.org/abs/1409.1081

\bibitem{Nurmiev2000}
A. G. Nurmiev. Orbits and invariants of third-order matrices, {\em Mat. Sb.} 191 (2000) 101-108.
\bibitem{ShGoHa2012}
R. Shaw; N. Gordon; H. Havlicek. Aspects of the Segre variety $S_{1,1,1}(2)$. {\em Des. Codes Cryptogr.} 62 (2): 225-239, 2012.
\bibitem{Stavrou2013}
S.G. Stavrou. Canonical forms of $2\! \times\! 2\! \times \! 2$ and $2\! \times\! 2\! \times \! 2\! \times\!  2$ symmetric
tensors over prime fields, {\em Linear and Multilinear Algebra}, online
first, 2013.

\bibitem{ThCh1938}
R. M. Thrall; J. H. Chanler. Ternary trilinear forms in the field of complex numbers, {\em Duke Math. J.}, 4 (1938) 678-690.

\bibitem{Thrall1938}
R. M. Thrall. Metabelian groups and trilinear forms, {\em Amer. J. Math.}, 60 (1938) 383-415.
\end{thebibliography}
\end{document}